\documentclass[10pt,twoside,final]{siamltex}
\usepackage{amsfonts,epsfig}

\newtheorem{example}[theorem]{Example}

\newtheorem{conjecture}{Conjecture}

\newcommand{\reals}{\mathbb{R}}

\usepackage{amsmath}
\newcommand{\spd}[1]{\mathbb P_{#1}}
\DeclareMathOperator{\ssharp}{\#}
\newcommand{\Sym}[1]{\mathfrak S_{#1}}
\newcommand{\Di}[1]{\mathfrak D_{#1}}
\newcommand{\Alt}[1]{\mathfrak A_{#1}}
\DeclareMathOperator{\stab}{Stab}
\DeclareMathOperator{\rstab}{RStab}
\newcommand{\norm}[1]{\left\Vert#1\right\Vert}

\begin{document}



\bibliographystyle{plain}
\title{Constructing matrix geometric means
\thanks{Received by the editors on Month x, 200x.
Accepted for publication on Month y, 200y   Handling Editor: .}}

\author{Federico Poloni\thanks{Scuola Normale Superiore, Piazza dei Cavalieri, 7, 56126 Pisa, Italy (\texttt{f.poloni@sns.it}).}}

\pagestyle{myheadings}
\markboth{F. Poloni}{Constructing matrix means}
\maketitle

\begin{abstract}
 In this paper, we analyze the process of ``assembling'' new matrix geometric means from existing ones, through function composition or limit processes. We show that for $n=4$ a new matrix mean exists which is simpler to compute than the existing ones. Moreover, we show that for $n>4$ the existing proving strategies cannot provide a mean computationally simpler than the existing ones.
\end{abstract}



\section{Introduction}
\paragraph{Literature review} In the last few years, several papers have been devoted to defining a proper way to generalize the concept of geometric mean to $n \geq 3$ Hermitian, positive definite $m \times m$ matrices. A seminal paper by Ando, Li and Mathias \cite{alm} defined the mathematical problem by stating ten properties that a ``good'' matrix geometric mean should satisfy. However, these properties do not uniquely define a multivariate matrix geometric mean; thus several different definitions appeared in literature.

Ando, Li and Mathias \cite{alm} first proposed a mean whose definition for $n$ matrices is based on a limit process involving several geometric means of $n-1$ matrices. Later Bini, Meini and Poloni \cite{bmp-means} noted that the slow convergence speed of this method prevents its use in applications; its main shortcoming is the fact that its complexity grows as $O(n!)$ with the number of involved matrices. In the same paper, they proposed a similar limit process with increased convergence speed, but still with complexity $O(n!)$. P\'alfia \cite{palfia} proposed a mean based on a similar process involving only means of $2$ matrices, and thus much simpler and cheaper to compute, but lacking property P3 (permutation invariance) from the ALM list. Lim \cite{lim} proposed a family of matrix geometric means that are based on an iteration requiring at each step the computation of a mean of $m\geq n$ matrices. Since the computational complexity for all known means greatly increases with $n$, the resulting family is useful as an example but highly impractical for numerical computations.

At the same time, Moakher \cite{moakher-simax,moakher} and Bhatia and Holbrook \cite{bhatiahol,bhatiabook} proposed a completely different definition, which we shall call the \emph{Riemannian centroid} of $A_1,A_2,\dots,A_n$. The Riemannian centroid $G^R(A_1,A_2,\dots,A_n)$ is defined as the minimizer of a sum of squared distances,
\begin{equation}\label{cartan1}
 G^{R}(A_1,A_2,\dots,A_n)=\arg\min_X \sum_{i=1}^n \delta^2(A_i,X),
\end{equation}
where $\delta$ is the geodesic distance induced by a natural Riemannian metric on the space of symmetric positive definite matrices. The same $X$ is the unique solution of the equation
\begin{equation}\label{cartan2}
 \sum_{i=1}^n \log(A_i^{-1}X)=0,
\end{equation}
involving the matrix logarithm function. While most of the ALM properties are easy to prove, it is still an open problem whether it satisfies P4 (monotonicity). The computational experiments performed up to now gave no counterexamples, but the monotonicity of the Riemannian centroid is still a conjecture \cite{bhatiahol}, up to our knowledge.

Moreover, while the other means had constructive definitions, it is not apparent how to compute the solution to either \eqref{cartan1} or \eqref{cartan2}. Two methods have been proposed, one based on a fixed-point iteration \cite{moakher} and one on the Newton methods for manifolds \cite{palfia,moakher}. Although both seem to work well on ``tame'' examples, their computational results show a fast degradation of the convergence behavior as the number of matrices and their dimension increase. It is unclear whether on more complicated examples there is convergence in the first place; unlike the other means, the convergence of these iteration processes has not been proved, as far as we know.

\paragraph{Notations}
Let us denote by $\spd{m}$ the space of Hermitian positive-definite $m \times m$ matrices. For all $A,B \in \spd{m}$, we shall say that $A<B$ ($A\leq B$) if $B-A$ is positive definite (semidefinite). With $A^*$ we denote the conjugate transpose of $A$.
We shall say that $\underline A=(A_i)_{i=1}^n \in (\spd{m})^n$ is a \emph{scalar} $n$-tuple of matrices if $A_1=A_2=\dots=A_n$. We shall use the convention that both $Q(\underline A)$ and $Q(A_1,\dots,A_n)$ denote the application of the map $Q:(\spd{m})^n \to \spd{m}$ to the $n$-tuple $\underline A$.

\paragraph{ALM properties}
Ando, Li and Mathias \cite{alm} introduced ten properties defining when a map $G:(\spd{m})^n \to \spd{m}$ can be called a \emph{geometric mean}. Following their paper, we report here the properties for $n=3$ only, for the sake of simplicity; the generalization to different values of $n$ is straightforward.
\begin{description}
 \item [P1] (consistency with scalars) If $A$, $B$, $C$ commute then $G(A,B,C)=(ABC)^{1/3}$.
 \item [P1'] This implies $G(A,A,A)=A$.
 \item [P2] (joint homogeneity) $G(\alpha A, \beta B, \gamma C)=(\alpha\beta\gamma)^{1/3}G(A,B,C)$, for each $\alpha, \beta, \allowbreak \gamma > 0$.
 \item [P2'] This implies $G(\alpha A, \alpha B, \alpha C)=\alpha G(A,B,C)$.
 \item [P3] (permutation invariance) $G(A,B,C)=G(\pi(A,B,C))$ for all the permutations $\pi(A,B,C)$ of $A$, $B$, $C$.
 \item [P4] (monotonicity) $G(A,B,C) \geq G(A',B',C')$ whenever $A \geq A'$, $B \geq B'$, $C \geq C'$.
 \item [P5] (continuity from above) If $A_n$, $B_n$, $C_n$ are monotonic decreasing sequences converging to $A$, $B$, $C$, respectively, then $G(A_n,B_n,C_n)$ converges to $G(A,B,C)$.
 \item [P6] (congruence invariance) $G(S^*AS, S^*BS, S^*CS)=S^*G(A,B,C)S$ for any nonsingular $S$.
 \item [P7] (joint concavity) If $A=\lambda A_1+ (1-\lambda)A_2$, $B=\lambda B_1+ (1-\lambda)B_2$, $C=\lambda C_1+ (1-\lambda)C_2$, then $G(A,B,C) \geq \lambda G(A_1,B_1,C_1)+(1-\lambda)G(A_2,B_2,C_2)$.
 \item [P8] (self-duality) $G(A,B,C)^{-1}=G(A^{-1},B^{-1},C^{-1})$.
 \item [P9] (determinant identity) $\det G(A,B,C)=(\det A \det B \det C)^{1/3}$.
 \item [P10] (arithmetic--geometric--harmonic mean inequality)
\[
 \frac{A+B+C}3 \geq G(A,B,C) \geq \left(\frac{A^{-1}+B^{-1}+C^{-1}}3 \right)^{-1}.
\]
\end{description}

\paragraph{The matrix geometric mean for $n=2$}
For $n=2$, the ALM properties uniquely define a matrix geometric mean which can be expressed explicitly as
\begin{equation}\label{defssharp}
 A \ssharp B := A(A^{-1}B)^{1/2}.
\end{equation}
This is a particular case of the more general map
\begin{equation}\label{defssharpt}
 A \ssharp_t B := A(A^{-1}B)^t, \quad t \in \reals,
\end{equation}
which has a geometrical interpretation as the parametrization of the geodesic joining $A$ and $B$ for a certain Riemannian geometry on $\spd{m}$ \cite{bhatiabook}.

\paragraph{The ALM and BMP means} Ando, Li and Mathias \cite{alm} recursively define a matrix geometric mean $G^{ALM}_n$ of $n$ matrices in this way. The mean $G^{ALM}_2$ of two matrices coincides with \eqref{defssharp}; for $n \geq 3$, suppose the mean of $n-1$ matrices $G^{ALM}_{n-1}$ is already defined. Given $A_1,\dots,A_n$, compute for each $j=1,2,\dots$
\begin{equation}\label{defalm}
 A_i^{(j+1)}:=G^{ALM}_{n-1}(A_1^{(j)}, A_2^{(j)}, \dots A_{i-1}^{(j)}, A_{i+1}^{(j)},\dots A_n^{(j)}) \quad i=1,\dots,n,
\end{equation}
where $A^{(0)}_i:=A_i$, $i=1,\dots n$.
The sequences $(A^{(j)}_i)_{j=1}^{\infty}$ converge to a common (not depending on $i$) matrix, and this matrix is a geometric mean of $A^{(0)}_1,\dots,A^{(0)}_n$.

The mean proposed by Bini, Meini and Poloni \cite{bmp-means} is defined in the same way, but with \eqref{defalm}  replaced by
\begin{equation}\label{defbmp}
 A_i^{(j+1)}:=G^{BMP}_{n-1}(A_1^{(j)}, A_2^{(j)}, \dots A_{i-1}^{(j)}, A_{i+1}^{(j)},\dots A_n^{(j)}) \ssharp_{1/n} A_i  \quad i=1,\dots,n.
\end{equation}
Though both maps satisfy the ALM properties, matrices $A,B,C$ exist for which $G^{ALM}(A,B,C) \neq G^{BMP}(A,B,C)$.

While the former iteration converges linearly, the latter converges cubically, and thus allows one to compute a matrix geometric mean with a lower number of iterations. In fact, if we call $p_k$ the average number of iterations that the process giving a mean of $k$ matrices takes to converge (which may vary significantly depending on the starting matrices), the total computational cost of the ALM and BMP means can be expressed as $O(n! p_3 p_4 \dots p_n m^3)$. The only difference between the two complexity bounds lies in the expected magnitude of the values $p_k$. The presence of a factorial and of a linear number of factors $p_k$ is undesirable, since it means that the problem scales very badly with $n$. In fact, already with $n=7,8$ and moderate values of $m$, a large CPU time is generally needed to compute a matrix geometric mean \cite{bmp-means}.

\paragraph{The P\'alfia mean}
P\'alfia \cite{palfia} proposed to consider the following iteration. Let again $A^{(0)}_i:=A_i,\,i=1,\dots, n$. Let us define
\begin{equation}\label{palfiamean}
 A^{(k+1)}_i:=A^{(k)}_i \ssharp A^{(k)}_{i+1},\quad i=1,\dots,n,
\end{equation}
where the indices are taken modulo $n$, i.e., $A^{(k)}_{n+1}=A^{(k)}_1$ for all $k$. We point out that the definition in the original paper \cite{palfia} is slightly different, as it considers several possible orderings of the input matrices, but the means defined there can be put in the form \eqref{palfiamean} up to a permutation of the starting matrices $A_1,\dots,A_n$.

As for the previous means, it can be proved that the iteration \eqref{palfiamean} converges to a scalar $n$-tuple; we call the common limit of all components $G^P(A_1,\dots,A_n)$. As we noted above, this function does not satisfy P3 (permutation invariance), and thus it is not a geometric mean in the ALM sense.

\paragraph{Other composite means} Apart from the Riemannian centroid, all the other definitions follow the same pattern:
\begin{itemize}
 \item build new functions of $n$ matrices by taking nested compositions of the existing means---preferably using only means of less than $n$ matrices;
 \item take the common limit of a set of $n$ functions defined as in the above step.
\end{itemize}
The possibilities for defining new iterations following this pattern are endless. Ando, Li, Mathias, and Bini, Meini, Poloni chose to use in the first step composite functions using computationally expensive means of $n-1$ matrices; this led to poor convergence results. P\'alfia chose instead to use more economical means of two variables as starting points; this led to better convergence (no $O(n!)$), but to a function which is not symmetric with respect to permutations of its entries (P3, permutation invariance).

As we shall see in the following, the property P3 is crucial: all the other ones are easily proved for a mean defined as composition/limit of existing means.

A natural question to ask is whether we can build a matrix geometric mean of $n$ matrices as the composition of matrix means of less matrices, without the need of a limit process. Two such unsuccessful attempts are reported in the paper by Ando, Li and Mathias \cite{alm}, as examples of the fact that it is not easy to define a matrix satisfying P1--P10. The first is 
\begin{equation}\label{g4rec}
 G^{4rec}(A,B,C,D):=(A \ssharp B) \ssharp (C \ssharp D).
\end{equation}
Unfortunately, there are matrices such that $(A \ssharp B) \ssharp (C\ssharp D)\neq (A \ssharp C) \ssharp (B \ssharp D)$, so P3 fails. A second attempt is
\begin{equation}\label{grec}
 G^{rec}(A,B,C):=(A^{4/3}\ssharp B^{4/3}) \ssharp C^{2/3},
\end{equation}
where the exponents are chosen so that P1 (consistency with scalars) is satisfied. Again, this function is not symmetric in its arguments, and thus fails to satisfy P3.

A second natural question is whether an iterative scheme such as the ones for $G^{ALM}$, $G^{BMP}$ and $G^P$ can yield P3 without having a $O(n!)$ computational cost. For example, if we could build a scheme similar to the ALM and BMP ones, but using only means of $\frac{n}{2}$ matrices in the recursion, then the $O(n!)$ growth would disappear. 

In this paper, we aim to analyze in more detail the process of ``assembling'' new matrix means from the existing ones, and show which new means can be found, and what cannot be done because of group-theoretical obstructions related to the symmetry properties of the composed functions. By means of a group-theoretical analysis, we will show that for $n=4$ a new matrix mean exists which is simpler to compute than the existing ones; numerical experiments show that the new definition leads to a significant computational advantage. Moreover, we will show that for $n>4$ the existing strategies of composing matrix means and taking limits cannot provide a mean which is computationally simpler than the existing ones.

\section{Quasi-means and notation}

\paragraph{Quasi-means}
Let us introduce the following variants to some of the Ando--Li--Mathias properties.
\begin{description}
 \item [P1''] Weak consistency with scalars. There are $\alpha, \beta, \gamma \in \reals$ such that if $A,B,C$ commute, then $G(A,B,C)=A^\alpha B^\beta C^\gamma$.
 \item[P2''] Weak homogeneity. There are $\alpha, \beta, \gamma \in \reals$ such that for each $r,s,t>0$, $G(rA,sB,tC)=r^\alpha s^\beta t^\gamma G(A,B,C)$.  Notice that if P1'' holds as well, these must be the same $\alpha, \beta, \gamma$ (proof: substitute scalar values in P1'').
 \item[P9'] Weak determinant identity. For all $d>0$, if $\det A=\det B=\det C=d$, then $\det G(A,B,C)=d$.
\end{description}

We shall call a \emph{quasi-mean} a function $Q:(\spd{m})^n \to (\spd{m})$ that satisfies P1'',P2'', P4, P6, P7, P8, P9'. This models expressions which are built starting from basic matrix means but are not symmetric, e.g., $A \ssharp G(B, C, D \ssharp E)$, \eqref{g4rec}, and \eqref{grec}.

\begin{theorem}
 If a quasi-mean $Q$ satisfies P3 (permutation invariance), then it is a geometric mean.
\end{theorem}
\begin{proof}
 From P2'' and P3, it follows that $\alpha=\beta=\gamma$. From P9', it follows that if $\det A=\det B = \det C=1$,
 \[
  2^m=\det Q(2A,2B,2C)=\det\left(2^{\alpha+\beta+\gamma} Q(A,B,C)\right)=2^{m(\alpha+\beta+\gamma)},
 \]
thus $\alpha+\beta+\gamma=1$. The two relations combined together yield $\alpha=\beta=\gamma=1/3$. Finally, it is proved in Ando, Li and Mathias \cite{alm} that P5 and P10 are implied by the other eight properties P1--P4 and P6--P9.
\end{proof}

For two quasi-means $Q$ and $R$ of $n$ matrices, we shall write $Q=R$ if $Q(\underline A)=R(\underline A)$ for each $n$-tuple $\underline A \in \spd{m}$

\paragraph{Group theory notation}
The notation $H \leq G$ ($H < G$) means that $H$ is a subgroup (proper subgroup) of $G$. Let us denote by $\Sym{n}$ the symmetric group on $n$ elements, i.e., the group of all permutations of the set $\{1,2,\dots,n\}$. As usual, the symbol $(a_1 a_2 a_3 \dots a_k)$ stands for the permutation (``cycle'') that maps $a_1 \mapsto a_2$, $a_2 \mapsto a_3$, \dots $a_{k-1} \mapsto a_k$, $a_k \mapsto a_1$ and leaves the other elements of $\{1,2,\dots n\}$ unchanged. Different symbols in the above form can be chained to denote the group operation of function composition; for instance, $\sigma=(13)(24)$ is the permutation $(1,2,3,4)\mapsto (3,4,1,2)$. We shall denote by $\Alt{n}$ the alternating group on $n$ elements, i.e., the only subgroup of index 2 of $\Sym{n}$, and by $\Di{n}$ the dihedral group over $n$ elements, with cardinality $2n$. The latter is identified with the subgroup of $\Sym{n}$ generated by the rotation $(1,2,\dots,n)$ and the mirror symmetry $(2,n)(3,n-1)\cdots(n/2,n/2+2)$ (for even values of $n$) or $(2,n)(3,n-1)\cdots((n+1)/2,(n+3)/2)$ (for odd values of $n$).

\paragraph{Coset transversals}
Let now $H \leq \Sym{n}$, and let $\{\sigma_1,\dots, \sigma_r\} \subset \Sym{n}$ be a transversal for the right cosets  $H\sigma$, i.e., a set of maximal cardinality $r=n!/|H|$ such that $\sigma_j\sigma_i^{-1} \not \in H$ for all $i \neq j$. The group $\Sym{n}$ acts by permutation over the cosets $(H\sigma_1,\dots, H\sigma_r)$, i.e., for each $\sigma$ there is a permutation $\tau=\rho_H(\sigma)$ such that
\[
 (H\sigma_1\sigma,\dots,H\sigma_r\sigma)=(H\sigma_{\tau(1)},\dots,H\sigma_{\tau(r)}).
\]
It is easy to check that in this case $\rho_H:\Sym{n}\to\Sym{r}$ must be a group homomorphism. Notice that if $H$ is a normal subgroup of $\Sym{n}$, then the action of $\Sym{n}$ over the coset space is represented by the quotient group $\Sym{n}/H$, and the kernel of $\rho_H$ is $H$. 

\begin{example}
The coset space of $H=\Di{4}$ has size $4!/8=3$, and a possible transversal is $\sigma_1=e$, $\sigma_2=(12)$, $\sigma_3=(14)$. We have $\rho_H(\Sym{4}) \cong \Sym{3}$: indeed, the permutation $\sigma=(12)\in\Sym{4}$ is such that $(H\sigma_1\sigma,H\sigma_2\sigma,H\sigma_3\sigma)=(H\sigma_2,H\sigma_1,H\sigma_3)$, and therefore $\rho_H(\sigma)=(12)$, while the permutation $\tilde\sigma=(14)\in\Sym{4}$ is such that $(H\sigma_1\tilde\sigma,H\sigma_2\tilde\sigma,H\sigma_3\tilde\sigma)=(H\sigma_3,H\sigma_2,H\sigma_1)$, therefore $\rho_H(\tilde\sigma)=(13)$. Thus $\rho_H(\Sym{4})$ must be a subgroup of $\Sym{3}$ containing $(12)$ and $(13)$, that is, $\Sym{3}$ itself.
\end{example}

With the same technique, noting that $\sigma_i^{-1}\sigma_j$ maps the coset $H\sigma_i$ to $H\sigma_j$, we can prove that the action $\rho_H$ of $\Sym{n}$ over the coset space is transitive.

\paragraph{Group action and composition of quasi-means} We may define a right action of $\Sym{n}$ on the set of quasi-means of $n$ matrices as
\[
(Q\sigma)(A_1,\dots,A_n):=Q(A_{\sigma(1)},\dots,A_{\sigma(n)}).
\]
The choice of putting $\sigma$ to the right, albeit slightly unusual, was chosen to simplify some of the notations used in Section~\ref{s:limits}.

When $Q$ is a quasi-mean of $r$ matrices  and $R_1, R_2, \dots R_r$ are quasi-means of $n$ matrices, let us define $Q \circ (R_1, R_2, \dots R_r)$ as the map
\begin{equation}\label{defcomp}
 \left(Q \circ (R_1, R_2, \dots R_r)\right) (\underline A) := Q(R_1(\underline A), R_2(\underline A), \dots, R_r(\underline A)).
\end{equation}
\begin{theorem} \label{t:composition}
Let $Q(A_1,\dots,A_r)$ and $R_j(A_1,\dots A_n)$ (for $j=1,\dots,r$) be quasi-means. Then,
\begin{enumerate}
 \item For all $\sigma \in \Sym{r}$, $Q\sigma$ is a quasi-mean.
 \item $(A_1,\dots,A_r,A_{r+1}) \mapsto Q(A_1,\dots,A_r)$ is a quasi-mean.
 \item $Q \circ (R_1, R_2, \dots R_r)$ is a quasi-mean.
\end{enumerate}
\end{theorem}
\begin{proof}
 All properties follow directly from the monotonicity (P4) and from the corresponding properties for the means $Q$ and $R_j$.
\end{proof}

We may then define the \emph{isotropy group}, or \emph{stabilizer group} of a quasi-mean $Q$
\begin{equation}\label{defstab}
 \stab(Q) := \{ \sigma \in \mathfrak S^n : Q= Q\sigma\}.
\end{equation}

\section{Means obtained as map compositions}\label{s:compo}

\paragraph{Reductive symmetries}
Let us define the concept of \emph{reductive symmetries} of a quasi-mean as follows. 
\begin{itemize}
\item in the special case in which $G_2(A,B)=A\ssharp B$, the symmetry property that $A\ssharp B=B\ssharp A$ is a reductive symmetry.
\item let $Q\circ(R_1,\dots,R_r)$ be a quasi-mean obtained by composition. The symmetry with respect to the permutation $\sigma$ (i.e., the fact that $Q=Q\sigma$) is a \emph{reductive symmetry} for $Q\circ(R_1,\dots,R_r)$ if this property can be formally proved relying only on the reductive symmetries of $Q$ and $R_1,\dots, R_r$.
\end{itemize}

For instance, if we take $Q(A,B,C):=A\ssharp (B\ssharp C)$, then we can deduce that $Q(A,B,C)=Q(A,C,B)$ for all $A,B,C$, but not that $Q(A,B,C)=Q(B,C,A)$ for all $A,B,C$. This does not imply that such a symmetry property does not hold: if we were considering the operator $+$ instead of $\ssharp$, then it would hold that $A+B+C=B+C+A$, but there are no means of proving it relying only on the commutativity of addition --- in fact, associativity is crucial.

As we stated in the introduction, Ando, Li and Mathias \cite{alm} showed explicit counterexamples proving that all the symmetry properties of $G^{4rec}$ and $G^{rec}$ are reductive symmetries. We conjecture the following.
\begin{conjecture}\label{ass}
All the symmetries of a quasi-mean obtained by recursive composition from $G_2$ are reductive symmetries.
\end{conjecture}

In other words, we postulate that no ``unexpected symmetries'' appear while examining quasi-means compositions. This is a rather strong statement; however, the numerical experiments and the theoretical analysis performed up to now never showed any invariance property that could not be inferred by those of the underlying means.

We shall prove several result limiting the reductive symmetries that a mean can have; to this aim, we introduce the \emph{reductive isotropy group}
\begin{equation}
 \rstab(Q)=\{\sigma \in \stab(Q) : \text{$Q=Q\sigma$ is a reductive symmetry}\}.
\end{equation}
We will prove that there is no quasi-mean $Q$ such that $\rstab(Q)=\Sym{n}$. This shows that the existing ``tools'' in the mathematician's ``toolbox'' do not allow one to construct a matrix geometric mean (with full proof) based only on map compositions; thus we need either to devise a completely new construction or to find a novel way to prove additional invariance properties involving map compositions.

\paragraph{Reduction to a special form}
The following results show that when looking for a reductive matrix geometric mean, i.e., a quasi-mean $Q$ with $\rstab{Q}=\Sym{n}$, we may restrict our search to quasi-means of a special form.
\begin{theorem}\label{th:specialform}
 Let $Q$ be a quasi-mean of $r+s$ matrices, and $R_1, R_2, \dots, R_r$, $S_1, S_2,\dots, S_s$ be quasi-means of $n$ matrices such that $R_i\neq S_j\sigma $ for all $i,j$ and every $\sigma \in \Sym{n}$. Then, 
\begin{multline}
\rstab(Q \circ (R_1,R_2,\dots, R_r,S_1,S_2,\dots, S_s))\\  \subseteq \rstab(Q \circ (R_1,\dots, R_r,R_1,R_1,\dots, R_1)).
\end{multline}
\end{theorem}
\begin{proof}
Let $\sigma \in \rstab(Q \circ (R_1,R_2,\dots, R_r,S_1,S_2,\dots S_s))$; since the only invariance properties that we may assume on $Q$ are those predicted by its invariance group, it must be the case that
\[
 (R_1\sigma, R_2\sigma, \dots, R_r\sigma, S_1\sigma ,  S_2\sigma, \dots S_s\sigma )
\]
is a permutation of $(R_1, R_2, \dots, R_r, S_1, S_2, \dots S_s)$ belonging to $\rstab(Q)$. Since $R_i\neq S_j\sigma $, this permutation must map the sets $\{R_1, R_2, \dots, R_r\}$ and $\{S_1, S_2, \dots S_s\}$ to themselves. Therefore, the same permutation maps
\[
(R_1, R_2, \dots, R_r, R_1, R_1, \dots R_1)
\]
 to 
\[
 (R_1\sigma , R_2\sigma , \dots, R_r\sigma , R_1\sigma, R_1\sigma , \dots R_1\sigma ).
\]
This implies that
\begin{multline*}
 Q(R_1, R_2, \dots, R_r, R_1, R_1, \dots R_1)
 =Q(R_1\sigma, R_2\sigma , \dots, R_r\sigma,  R_1\sigma, R_1\sigma, \dots R_1\sigma)
\end{multline*}
as requested.
\end{proof}

\begin{theorem}
 Let $M_1:=Q\circ(R_1,R_2,\dots,R_r)$ be a quasi-mean. Then there is a quasi-mean $M_2$ in the form 
\begin{equation}\label{specform}
\tilde Q \circ (\tilde R\sigma_1,\tilde R\sigma_2,\dots,\tilde R\sigma_{\tilde r}),  
\end{equation}
 where $(\sigma_1,\sigma_2,\dots,\sigma_{\tilde r})$ is a right coset transversal for $\rstab(\tilde R)$ in $\Sym{n}$, such that $\rstab(M_1) \subseteq \rstab(M_2)$.
\end{theorem}
\begin{proof}
Set $\tilde R=R_1$. For each $i=2,3,\dots,r$ if $R_i\neq \tilde R\sigma$, we may replace it with $\tilde R$, and by Theorem~\ref{th:specialform} the restricted isotropy group increases or stays the same. Thus by repeated application of this theorem, we 
may reduce to the case in which each $R_i$ is in the form $\tilde R\tau_i$ for some permutation $\tau_i$.

Since $\{\sigma_i\}$ is a right transversal, we may write $\tau_i=h_i \sigma_{k(i)}$ for some $h_i\in H$ and $k(i)\in\{1,2,\dots,\tilde r\}$. We have $\tilde R h=\tilde R$ since $h\in\stab\tilde R$, thus $R_i=\tilde R\sigma_{k(i)}$. The resulting quasi-mean is $Q\circ(\tilde R \sigma_{k(1)},\dots,\tilde R\sigma_{k(r)})$. Notice that we may have $k(i)=k(j)$, or some cosets may be missing. Let now $\tilde Q$ be defined as $\tilde Q(A_1,A_2,\dots,A_{\tilde r}):=Q(A_{k(1)},\dots,A_{k(r)})$; then we have
\begin{equation}\label{uguagl}	
 \tilde Q(\tilde R\sigma_1,\dots,\tilde R\sigma_{\tilde r})=
 Q(\tilde R\sigma_{k(1)},\dots,\tilde R\sigma_{k(r)})
\end{equation}
and thus the isotropy groups of the left-hand side and right-hand side coincide.
\end{proof}

For the sake of brevity, we shall define
\[
Q \circ R:=Q\circ(R\sigma_1,\dots,R\sigma_r),
\]
assuming a standard choice of the transversal for $H=\stab R$. Notice that $Q \circ R$ depends on the ordering of the cosets $H\sigma_1, \dots, H\sigma_r$, but not on the choice of the coset representative $\sigma_i$, since $Qh\sigma_i = Q\sigma_i$ for each $h\in H$.

\begin{example}
 The quasi-mean $(A,B,C) \mapsto (A \ssharp B) \ssharp (B \ssharp C)$ is $Q \circ Q$, where $Q(X,Y,Z)=X \ssharp Y$,  $H=\{e,(12)\}$, and the transversal is $\{e,(13),(23)\}$.
\end{example}
\begin{example}
The quasi-mean $(A,B,C) \mapsto (A \ssharp B) \ssharp C$ is not in the form \eqref{specform}, but in view of Theorem~\ref{th:specialform}, its restricted isotropy group is a subgroup of that of $(A,B,C) \mapsto (A \ssharp B) \ssharp (A \ssharp B)$.
\end{example}

The following theorem shows which permutations we can actually prove to belong to the reductive isotropy group of a mean in the form \eqref{specform}.

\begin{theorem}\label{t:isosharp}
Let $H \leq \Sym{n}$, $R$ be a quasi-mean of $n$ matrices such that $\rstab{R}=H$ and $Q$ be a quasi-mean of $r=n!/|H|$ matrices. Let $G\in\Sym{n}$ be the largest permutation subgroup such that $\rho_H(G) \leq \rstab(Q)$.
Then, $G = \rstab(Q \circ R)$.
\end{theorem}
\begin{proof}
 Let $\sigma \in G$ and $\tau=\rho_H(\sigma)$; we have
\begin{equation*}
\begin{split}
 (Q \circ R)\sigma(\underline A)&  =  Q\bigl(R\sigma_1 \sigma(\underline A),R\sigma_2 \sigma(\underline A), \dots, R\sigma_r \sigma(\underline A)\bigr)\\
&=Q\bigl(R\sigma_{\tau(1)} (\underline A),R\sigma_{\tau(2)} (\underline A), \dots, R\sigma_{\tau(r)}(\underline A)\bigr)\\
&=Q\bigl(R\sigma_1(\underline A),R\sigma_2 (\underline A), \dots, R\sigma_r (\underline A)\bigr),
\end{split}
\end{equation*}
where the last equality holds because $\tau \in \stab(Q)$.

Notice that the above construction is the only way to obtain invariance with respect to a given permutation $\sigma$: indeed, to prove invariance relying only on the invariance properties of $Q$, $(R\sigma_1\sigma,\dots,R\sigma_r\sigma)=(R\sigma_{\tau(1)},\dots,R\sigma_{\tau(r)})$ must be a permutation of $(R\sigma_1,\dots,R\sigma_r)$ belonging to $\rstab Q$, and thus $\rho_H(\sigma)=\tau \in \stab Q$. Thus the reductive invariance group of the composite mean is precisely the largest subgroup $G$ such that $\rho_H(G)\leq \stab Q$.
\end{proof}

\begin{example}\label{tournamentmean}
Let $n=4$, $Q$ be any (reductive) geometric mean of three matrices (i.e., $\rstab Q=\Sym{3}$), and $R(A,B,C,D):=(A \ssharp C) \ssharp (B \ssharp D)$. We have $H=\rstab R=\Di{4}$, the dihedral group over four elements, with cardinality 8. There are $r=4!/|H|=3$ cosets. Since $\rho_H(\Sym{4})$ is a subset of $\stab Q=\Sym{3}$, the isotropy group of $Q\circ R$ contains $G=\Sym{4}$ by Theorem~\ref{t:isosharp}. Therefore $Q\circ R$ is a geometric mean of four matrices.
\end{example}

Indeed, the only assertion we have to check explicitly is that $\rstab R=\Di{4}$. The isotropy group of $R$ contains $(24)$ and $(1234)$, since by using repeatedly the fact that $\ssharp$ is symmetric in its arguments we can prove that $R(A,B,C,D)=R(A,D,C,B)$ and $R(A,B,C,D)=R(D,A,B,C)$. Thus it must contain the subgroup generated by these two elements, that is, $\Di{4}\leq\rstab R$. The only subgroups of $\Sym{4}$ containing $\Di{4}$ as a subgroup are the two trivial ones $\Sym{4}$ and $\Di{4}$. We cannot have $\rstab R=\Sym{4}$, since $R$ has the same definition as $G^{4rec}$ of equation \eqref{g4rec}, apart from a reordering, and it was proved \cite{alm} that this is not a geometric mean.

It is important to notice that by choosing $G_3=G_3^{ALM}$ or $G_3=G_3^{BMP}$ in the previous example we may obtain a geometric mean of four matrices using a single limit process, the one needed for $G_3$. This is more efficient than $G^{ALM}_4$ and $G^{BMP}_4$, which compute a mean of four matrices via several means of three matrices, each of which requires a limit process in its computation. We will return to this topic in Section~\ref{s:exp}.

\paragraph{Above four elements}
Is it possible to obtain a reductive geometric mean of $n$ matrices, for $n>4$, starting from simpler means and using the construction of Theorem~\ref{t:isosharp}? The following result shows that the answer is no.

\begin{theorem}
Suppose $G :=\rstab(Q\circ R)\geq \Alt{n}$ and $n>4$. Then $\Alt{n} \leq \rstab(Q)$ or $\Alt{n} \leq \rstab(R)$.
\end{theorem}
\begin{proof}
 Let us consider $K=\ker \rho_H$. It is a normal subgroup of $\Sym{n}$, but for $n>4$ the only normal subgroups of $\Sym{n}$ are the trivial group $\{e\}$, $\Alt{n}$ and $\Sym{n}$ \cite{dixonmort}. Let us consider the three cases separately.
\begin{enumerate}
 \item $K=\{e\}$. In this case, $\rho_H(G) \cong G$, and thus $G \leq \rstab{Q}$.
 \item $K=\Sym{n}$. In this case, $\rho_H(\Sym{n})$ is the trivial group. But the action of $\Sym{n}$ over the coset space is transitive, since $\sigma_i^{-1}\sigma_j$ sends the coset $H\sigma_i$ to the coset $H\sigma_j$. So the only possibility is that there is a single coset in the coset space, i.e., $H=\Sym{n}$.
 \item $K=\Alt{n}$. As in the above case, since the action is transitive, it must be the case that there are at most two cosets in the coset space, and thus $H=\Sym{n}$ or $H=\Alt{n}$. 
\end{enumerate}
\end{proof}
Thus it is impossible to apply Theorem \ref{t:isosharp} to obtain a quasi-mean with reductive isotropy group containing $\Alt{n}$, unless one of the two starting quasi-means has a reductive isotropy group already containing $\Alt{n}$.

\section{Means obtained as limits}\label{s:limits}

\paragraph{An algebraic setting for limit means}
We shall now describe a unifying algebraic setting in terms of isotropy groups, generalizing the procedures leading to the means defined by limit processes $G^{ALM}$, $G^{BMP}$ and $G^P$.

Let $S: (\spd{m})^n \to (\spd{m})^n$ be a map; we shall say that $S$ \emph{preserves} a subgroup $H < \Sym{n}$ if there is a map $\tau: H \to H$ such that $Sh(\underline A)=\tau(h)S(\underline A)$ for all $\underline A \in \spd{m}$.

\begin{theorem}\label{th:iteration}
 Let $S: (\spd{m})^n \to (\spd{m})^n$ be a map and $H<\Sym{n}$ be a permutation group such that
\begin{enumerate}
 \item $(\underline A) \to \bigl(S(\underline A)\bigr)_i$ is a quasi-mean for all $i=1,\dots,n$,
 \item $S$ preserves $H$,
 \item for all $\underline A \in (\spd{m})^n$, $\lim_{k \to \infty} S^{k}(\underline A)$ is a scalar $n$-tuple\footnote{Here $S^k$ denotes function iteration: $S^1=S$ and $S^{k+1}(\underline A)=S(S^k(\underline A))$ for all $k$.},
\end{enumerate}
and let us denote by $S^{\infty}(\underline A)$ the common value of all entries of the scalar $n$-tuple $\lim_{k \to \infty} S^{k}(\underline A)$.
Then, $S^\infty(\underline A)$ is a quasi-mean with isotropy group containing $H$.
\end{theorem}
\begin{proof}
From Theorem \ref{t:composition}, it follows that $\underline A \mapsto \bigl(S^k(\underline A)\bigr)_i$ is a quasi-mean for each $k$. Since all the properties defining a quasi-mean pass to the limit, $S^\infty$ is a quasi-mean itself.

Let us take $h \in H$ and $\underline A \in \spd{n}$. It is easy to prove by induction on $k$ that $S^kh(\underline A)=\tau^k(h) \left(S^k(\underline A)\right)$. Now, choose a matrix norm inducing the Euclidean topology on $\spd{m}$; let $\varepsilon>0$ be fixed, and let us take $K$ such that for all $k> K$ and for all $i=1,\dots,n$ the following inequalities hold:
\begin{itemize}
 \item $\norm {\bigl(S^k(\underline A)\bigr)_i - S^{\infty}(\underline A)} < \varepsilon$,
\item $\norm {\bigl(S^kh(\underline A)\bigr)_i - S^{\infty}h(\underline A)} < \varepsilon$.
\end{itemize}
We know that $\bigl(S^kh(\underline A)\bigr)_i=\bigl(\tau^k(h)S^k(\underline A)\bigr)_i=\bigl(S^k(\underline A)\bigr)_{\tau^k(h)(i)}$, therefore
\begin{multline*}
 \norm {S^{\infty}(\underline A) - S^{\infty}h(\underline A)} \leq \norm {\bigl(S^k(\underline A)\bigr)_{\tau^k(h)(i)} - S^{\infty}(\underline A)}\\ + \norm {\bigl(S^kh(\underline A)\bigr)_i - S^{\infty}h(\underline A)} < 2\varepsilon.
\end{multline*}
Since $\varepsilon$ is arbitrary, the two limits must coincide. This holds for each $h \in H$, therefore $H \leq \stab{S^{\infty}}$.
\end{proof}

\begin{example}
 The map $S$ defining $G^{ALM}_4$ is
\[
  \begin{bmatrix}
   A \\ B \\ C \\D
  \end{bmatrix}
  \mapsto
  \begin{bmatrix}
   G_3^{ALM}(B,C,D) \\ G_3^{ALM}(A,C,D) \\ G_3^{ALM}(A,B,D) \\ G_3^{ALM}(A,B,C)
  \end{bmatrix}.
\]
One can see that $S\sigma=\sigma^{-1}S$ for each $\sigma \in \Sym{4}$, and thus with the choice $\tau(\sigma):=\sigma^{-1}$ we get that $S$ preserves $\Sym{4}$. Thus, by Theorem~\ref{th:iteration}, $S^\infty=G^{ALM}_4$ is a geometric mean of four matrices. The same reasoning applies to $G^{BMP}$.
\end{example}

\begin{example}
 The map $S$ defining $G^P_4$ is 
 \[
  \begin{bmatrix}
   A \\ B \\ C \\D
  \end{bmatrix}
  \mapsto
  \begin{bmatrix}
   A \ssharp B \\ B \ssharp C \\ C \ssharp D \\ D \ssharp A
  \end{bmatrix}.
 \]
$S$ preserves the dihedral group $\Di{4}$. Therefore, provided the iteration process converges to a scalar $n$-tuple, $S^{\infty}$ is a quasi-mean with isotropy group containing $\Di{4}$.
\end{example}

\paragraph{Efficiency of the limit process}
As in the previous section, we are interested in seeing whether this approach, which is the one that has been used to prove invariance properties of the known limit means \cite{alm,bmp-means}, can yield better results for a different map $S$.

\begin{theorem}
 Let $S: (\spd{m})^n \to (\spd{m})^n$ preserve a group $H$. Then, the invariance group of each of its components $S_i$, $i=1,\dots, n$, is a subgroup of $H$ of index at most $n$.
\end{theorem}
\begin{proof}
Let $i$ be fixed, and set $I_k:=\{h \in H : \tau(h)(i)=k\}$. The sets $I_k$ are mutually disjoint and their union is $H$, so the largest one has cardinality at least $|H|/n$, let us call it $I_{\bar k}$.

From the hypothesis that $S$ preserves $H$, we get $S_i h(\underline A)=S_{\bar k}(\underline A)$ for each $\underline A$ and each $h \in I_k$. Let $\bar h$ be an element of $I_k$; then $S_ih(\bar h^{-1} \underline A)=S_{\bar k}(\bar h^{-1}\underline A)=S_i(\underline A)$. Thus the isotropy group of $S_i$ contains all the elements of the form $h\bar h^{-1}$, $h \in I_k$, and those are at least $|H|/n$.
\end{proof}

The following result holds \cite[page 147]{dixonmort}.
\begin{theorem}
 For $n>4$, the only subgroups of $\Sym{n}$ with index at most $n$ are:
\begin{itemize}
 \item the alternating group $\Alt{n}$,
 \item the $n$ groups $T_k=\{\sigma \in \Sym{n}: \sigma(k)=k\}$, $k=1,\dots,n$, all of which are isomorphic to $\Sym{n-1}$,
 \item for $n=6$ only, another conjugacy class of $6$ subgroups of index $6$ isomorphic to $\Sym{5}$.
\end{itemize}
Analogously, the only subgroups of $\Alt{n}$ with index at most $n$ are:
\begin{itemize}
 \item the $n$ groups $U_k=\{\sigma \in \Alt{n}: \sigma(k)=k\}$, $k=1,\dots,n$, all of which are isomorphic to $\Alt{n-1}$,
 \item for $n=6$ only, another conjugacy class of $6$ subgroups of index $6$ isomorphic to $\Alt{5}$.
\end{itemize}
\end{theorem}

This shows that whenever we try to construct a geometric mean of $n$ matrices by taking a limit processes, such as in the Ando--Li--Mathias approach, the isotropy groups of the starting means must contain $\Alt{n-1}$. On the other hand, by Theorem~\ref{t:isosharp}, we cannot generate means whose isotropy group contains $\Alt{n-1}$ by composition of simpler means; therefore, there is no simpler approach than that of building a mean of $n$ matrices as a limit process of means of $n-1$ matrices (or at least quasi-means with $\stab Q=\Alt{n-1}$, which makes little difference). This shows that the recursive approach of $G^{ALM}$ and $G^{BMP}$ cannot be simplified while still maintaining P3 (permutation invariance).

\section{Computational issues and numerical experiments}\label{s:exp}
\paragraph{A faster mean of four matrices} The results we have exposed up to now are negative results, and they hold for $n>4$. On the other hand, it turns out that for $n=4$, since $\Alt{n}$ is not a simple group, there is the possibility of obtaining a mean that is computationally simpler than the ones in use. Such a mean is the one we described in Example~\ref{tournamentmean}. Let us take any mean of three elements (we shall use $G^{BMP}_3$ here since it is the one with the best computational results); the new mean is therefore defined as
\begin{multline}\label{defgnew}
 G^{NEW}_4(A,B,C,D) := G_3^{BMP}\left( (A \ssharp B) \ssharp (C \ssharp D), (A \ssharp C) \ssharp (B \ssharp D),\right.\\\left. (A \ssharp D) \ssharp (B \ssharp C) \right).
\end{multline}
Notice that only one limit process is needed to compute the mean; conversely, when computing $G_4^{ALM}$ or $G_4^{BMP}$ we are performing an iteration whose elements are computed by doing four additional limit processes; thus we may expect a large saving in the overall computational cost.

We may extend the definition recursively to $n>4$ elements using the construction described in \eqref{defbmp}, but with $G^{NEW}$ instead of $G^{BMP}$. The total computational cost, computed in the same fashion as for the ALM and BMP means, is $O(n! p_3 p_5 p_6 \dots p_n m^3)$. Thus the undesirable dependence from $n!$ does not disappear; the new mean should only yield a saving measured by a multiplicative constant in the complexity bound.

\paragraph{Benchmarks}
We have implemented the original BMP algorithm and the new one described in the above section with MATLAB\textregistered{} and run some tests on the same set of examples used by Moakher \cite{moakher} and Bini \emph{et al.} \cite{bmp-means}. It is an example deriving from physical experiments on elasticity. It consists of five sets of matrices to average, with $n$ varying from 4 to 6, and $6\times 6$ matrices split into smaller diagonal blocks.

For each of the five data sets, we have computed both the BMP and the new matrix mean. The CPU times are reported in Table~\ref{t:times}. As a stopping criterion for the iterations, we used
\[
 \max_{i,j,k} \left\vert(A^{(h)}_i)_{jk}-(A^{(h+1)}_i)_{jk}\right\vert < 10^{-13}.
\]
\begin{table}
\begin{tabular}{ccc}
Data set (number of matrices) & BMP mean & New mean \\
\hline
NaClO$_3$ (5) & 1.3E+00 & 3.1E-01\\
Ammonium dihydrogen phosphate (4) & 3.5E-01 &  3.9E-02\\
Potassium dihydrogen phosphate (4) & 3.5E-01 & 3.9E-02\\
Quartz (6) & 2.9E+01 & 6.7E+00\\
Rochelle salt (4) & 6.0E-01 & 5.5E-02
\end{tabular}
\caption{CPU times for the elasticity data sets}\label{t:times}
\end{table}
As we expected, our mean provides a substantial reduction of the CPU time which is roughly by an order of magnitude. 

Following Bini \emph{et al.} \cite{bmp-means}, we then focused on the second data set (ammonium dihydrogen phosphate) for a deeper analysis; we report in Table~\ref{t:hear} the number of iterations and matrix roots needed in both computations.
\begin{table}\label{t:hear}
\begin{small}
\begin{tabular}{ccc}
 & BMP mean & New mean\\
\hline
Outer iterations ($n=4$) & 3 & none\\
Inner iterations ($n=3$) & $4 \times 2.0$ (avg.) per outer iteration & 2\\
Matrix square roots (\texttt{sqrtm}) & 72 & 15\\
Matrix $p$-th roots (\texttt{rootm}) & 84 & 6\\
\end{tabular}
\end{small}
\caption{Number of inner and outer iterations needed, and number of matrix roots needed (ammonium dihydrogen phosphate)}
\end{table}

The examples in these data sets are mainly composed of matrices very close to each other; we shall consider here instead an example of mean of four matrices whose mutual distances are larger:
\begin{equation}\label{matricibuffe}
\begin{aligned}
 A&=\begin{bmatrix}
 1 &0& 0\\0 & 1 & 0\\ 0 & 0 & 1     
   \end{bmatrix},\!&
 B&=\begin{bmatrix}
    3 &0& 0\\0 & 4 & 0\\ 0 & 0 & 100
   \end{bmatrix},\!
&C&=\begin{bmatrix}
    2 &1& 1\\1 & 2 & 1\\ 1 & 1 & 2
   \end{bmatrix},\!&
D&=\begin{bmatrix}
    20 &0& -10\\0 & 20 & 0\\ -10 & 0 & 20
   \end{bmatrix}.
\end{aligned}
\end{equation}
The results regarding these matrices are reported in 
\begin{table}
\begin{small}
\begin{tabular}{ccc}
 & BMP mean & New mean\\
\hline
Outer iterations ($n=4$) & 4 & none\\
Inner iterations ($n=3$) & $4 \times 2.5$ (avg.) per outer iteration & 3\\
Matrix square roots (\texttt{sqrtm}) & 120 & 18\\
Matrix $p$-th roots (\texttt{rootm}) & 136 & 9\\
\end{tabular}
\end{small}
\caption{Number of inner and outer iterations needed, and number of matrix roots needed}\label{t:caso}
\end{table}
Table~\ref{t:caso}.

\paragraph{Accuracy} It is not clear how to check the accuracy of a limit process yielding a matrix geometric mean, since the exact value of the mean is not known \emph{a priori}, apart from the cases in which all the $A_i$ commute. In those cases, P1 yields a compact expression for the result. So we cannot test accuracy in the general case; instead, we have focused on two special examples.

As a first accuracy experiment, we computed $G(M^{-2},M,M^2,M^3)-M$, where $M$ is taken as the first matrix of the second data set on elasticity; the result of this computation should be zero according to P1. As a second experiment, we tested the validity of P9 (determinant identity) on the means of the four matrices in \eqref{matricibuffe}. The results of both computations are reported in Table~\ref{t:accuracy};
\begin{table}
\begin{center}
\renewcommand{\arraystretch}{1.2}
\begin{tabular}{cc}
Operation & Result\\
\hline
 $\left\Vert G^{BMP}(M^{-2},M,M^2,M^3)-M\right\Vert_2$ & 4.0E-14\\
 $\left\Vert G^{NEW}(M^{-2},M,M^2,M^3)-M\right\Vert_2$ & 2.5E-14\\
\hline
$\left\vert\det(G^{BMP}(A,B,C,D))-(\det(A)\det(B)\det(C)\det(D))^{1/4}\right\vert$ & 5.5E-13\\
$\left\vert\det(G^{BMP}(A,B,C,D))-(\det(A)\det(B)\det(C)\det(D))^{1/4}\right\vert$ & 2.1E-13
\end{tabular} 
\end{center}
\caption{Accuracy tests}\label{t:accuracy}
\end{table}
the results are well within the errors permitted by the stopping criterion, and show that both algorithms can reach a satisfying precision.

\section{Conclusions}
\paragraph{Research lines} The results of this paper show that, by combining existing matrix means, it is possible to create a new mean which is faster to compute than the existing ones. Moreover, we show that using only function compositions and limit processes with the existing proof strategies, it is not possible to achieve any further significant improvement with respect to the existing algorithms. In particular, the dependency from $n!$ cannot be removed. New attempts should focus on other aspects, such as:
\begin{itemize}
 \item proving new ``unexpected'' algebraic relations involving the existing matrix means, which would allow to break out of the framework of Theorem~\ref{t:isosharp}--Theorem~\ref{th:iteration}.
 \item introducing new kinds of matrix geometric means or quasi-means, different from the ones built using function composition and limits.
 \item proving that the Riemannian centroid \eqref{cartan1} is a matrix mean in the sense of Ando--Li--Mathias (currently P4 is an open problem), or providing faster and reliable algorithms to compute it.
\end{itemize}
It is an interesting question whether it is possible to construct a quasi-mean whose isotropy group is exactly $\Alt{n}$.

\paragraph{Acknowledgments} The author would like to thank Dario Bini and Bruno Iannazzo for enlightening discussions on the topic of matrix means, and Roberto Dvornicich and Francesco Veneziano for their help with the group theory involved in the analysis of the problem.


\end{document}